\documentclass[12pt]{article}
\usepackage{amsmath}
\usepackage{amsthm}
\usepackage{amssymb}
\usepackage[margin=1in]{geometry}
\usepackage{hyperref}
\usepackage{xcolor}
\usepackage{listings}
\usepackage{tikz}
\hypersetup{
    colorlinks,
    linkcolor={black},
    citecolor={black},
    urlcolor={black}
}

\usetikzlibrary{matrix}
\usetikzlibrary{chains}

\newtheorem{theorem}{Theorem}[section]

\newtheorem{lemma}[theorem]{Lemma}
\newtheorem{corollary}[theorem]{Corollary}

\theoremstyle{definition}
\newtheorem{definition}[theorem]{Definition}
\newtheorem{remark}[theorem]{Remark}
\newtheorem{example}[theorem]{Example}

\newcommand{\N}{\mathbb{N}}

\newcommand{\C}{{\mathbb{C}}}

\newcommand{\mA}{{\cal A}} 
\newcommand{\mB}{{\cal B}}
\newcommand{\mC}{{\cal C}} 
 
\newcommand{\mF}{{\cal F}}

\DeclareMathOperator{\rank}{rank}

\title{\large\bf Succinct Encodings of Binary Trees with Application to AVL Trees}
\author{Jeremy Chizewer, Stephen Melczer, J. Ian Munro, and Ava Pun }
\date{}
\begin{document}
\maketitle

\begin{abstract}
We use a novel decomposition to create succinct data structures -- supporting a wide range of operations on static trees in constant time -- for a variety tree classes, extending results of Munro, Nicholson, Benkner, and Wild. Motivated by the class of AVL trees, we further derive asymptotics for the information-theoretic lower bound on the number of bits needed to store tree classes whose generating functions satisfy certain functional equations. In particular, we prove that AVL trees require approximately $0.938$ bits per node to encode. 
\end{abstract}

Designing a data structure often requires a trade-off between the amount of space used to store objects in the structure and the time needed to compute with the objects. Generally speaking, increasing the information stored by a data structure makes it easier (and faster) to manipulate or determine properties of its objects. The use of tree-like structures to store information that can be efficiently searched goes back at least to the invention of the \emph{binary search tree} in the early 1960s~\cite{Windley1960,BoothColin1960,Hibbard1962}. Although simple to describe and implement, the worst-case search cost for a binary search tree with $n$ nodes is $O(n)$. Adelson-Velsky and Landis~\cite{Adelson-VelskiiLandis1962} proved that it is possible to maintain trees with logarithmic update cost and worst-case logarithmic search cost through the use of \emph{AVL trees}, a subclass of binary search trees with logarithmic height that is maintained with updates during insertions and deletions in logarithmic time. AVL trees are the oldest class of binary search trees maintaining logarithmic height, and are characterized by the key property that any pair of sibling subtrees differ in height by at most $1$. Since the introduction of AVL trees, many classes of balanced trees have been proposed as data structures, including the family of \emph{B-trees}~\cite{BayerMcCreight1970} and special cases of \emph{binary B-trees}~\cite{Bayer1971a} and \emph{red-black trees}~\cite{GuibasSedgewick1978}. Sedgewick~\cite{Sedgewick2008} introduced \emph{Left-Leaning Red Black Trees} to break symmetries and simplify implementation (see also Andersson~\cite{Andersson1993} for a similar strategy on binary B-trees). A modern study of balanced trees using a generalized notion of \emph{rank} can be found in Haeupler et al.~\cite{HaeuplerSenTarjan2015}.

\section{Succinct Data Structures and Tree Enumeration}

If $\mC=\bigcup_{n=0}^\infty\mC_n$ is a family of objects, with $\mC_n$ denoting the objects of size $n$ in $\mC$, then a representation of $\mC$ is called \emph{succinct} if it maps each object of $\mC_n$ to a unique string of length $\log_2 |\mC_n| + o(\log |\mC_n|)$. A succinct representation is thus one whose space complexity asymptotically equals, up to lower-order terms, the information-theoretic lower bound. A \emph{succinct data structure}~\cite{DBLP:reference/crc/Munro004, Navarro2016} for $\mC$ is a succinct representation of $\mC$ that supports a range of operations on $\mC$ under reasonable time constraints.

\subsection{Succinct Representations of Binary Search Trees}

To illustrate succinct data structures, let $\mB$ be the class of rooted planar binary trees, so that the number $|\mB_n|$ of objects in $\mB$ of size $n$ is the $n$th Catalan number $b_n = \frac{1}{n+1}\binom{2n}{n}$. As noted above, the class $\mB$ lends itself well to storing ordered data in a structure called a \emph{binary search tree}. The general idea is that for each node in the tree, the data stored in its left subtree will be smaller than the data at that node, and the data stored in the right subtree will be larger. To retrieve elements, one can recursively navigate through the tree by comparing the desired element to the current node, and moving to the left or right subtree if the element is respectively smaller or larger than the current node. As a result, it is desirable to efficiently support the navigation operations of moving to parent or child nodes in whatever representation is used.  

A naive representation of $\mB$ gives each node a label (using roughly $\log_2 n$ space) and stores the labels of each node's children and parent. The resulting data structure supports operations like finding node siblings in constant time, but is \emph{not} succinct as it uses $\Theta(n\log n)$ bits while the information-theoretic lower bound is only $\log_2(b_n) = 2n + o(n)$. Somewhat conversely, a naive space-optimal representation of $\mB$ is obtained by listing the objects of $\mB_n$ in any canonical order and referencing a tree by its position $\{1,\dots,b_n\}$ in the order, but asking for information like the children or parents of a node in a specific tree is then expensive as it requires building parts of the tree. 

\begin{figure}[t]\centering
\begin{tikzpicture}[level distance=32pt,
every node/.style={circle,draw,inner sep=1pt, minimum size=1.5em, fill=black},
dummy/.style={rectangle,minimum width=1em, fill=white},
level 1/.style={sibling distance=200pt},
level 2/.style={sibling distance=100pt},
level 3/.style={sibling distance=50pt},
level 4/.style={sibling distance=25pt}
]
\node[fill=black] {{\textcolor{white}{1}}}
    child {node {\textcolor{white}{2}} 
        child {node {\textcolor{white}{4}} 
            child {node {\textcolor{white}{8}} 
                child {node[dummy] {12}}
                child {node[dummy] {13}}
            }
            child {node {\textcolor{white}{9}} 
                child {node {\textcolor{white}{14}} 
                    child {node[dummy] {18}}
                    child {node[dummy] {19}}
                }
                child {node[dummy] {15}}
            }
        }
        child {node[dummy] {5}}
    }
    child {node {\textcolor{white}{3}} 
        child {node[dummy] {6}}
        child {node {\textcolor{white}{7}} 
            child {node {\textcolor{white}{10}} 
                child {node[dummy] {16}}
                child {node[dummy] {17}}
            }
            child {node[dummy] {11}}
        }
    }
;
\end{tikzpicture}

\bigskip

\begin{tikzpicture}[>=stealth, start chain, node distance=-0.5pt]
\node[on chain, draw, style=circle, fill=black, minimum size=1.2em]{};
\node[on chain] (h) {= original node};
\node[on chain] (h) {};
\node[on chain] (h) {};
\node[on chain, draw, minimum size=1.0em]{};
\node[on chain] (h) {= external node};
\end{tikzpicture}

\medskip

\begin{tikzpicture}[
list/.style={rectangle, draw, minimum width=1.5em, minimum height=14pt}, >=stealth, start chain, node distance=-0.5pt,
every label/.style={draw=none,minimum width=0pt,text=blue,font=\tt\footnotesize},
label distance=0mm,
]
\node[on chain] (h) {Level-order bitmap:};
\node[on chain] (h) {};
\node[list,on chain,label={above:1}] (A) {1};
\node[list,on chain,label={above:2}] (B) {1};
\node[list,on chain,label={above:3}] (C) {1};
\node[list,on chain,label={above:4}] (D) {1};
\node[list,on chain,label={above:5}] (E) {0};
\node[list,on chain,label={above:6}] (F) {0};
\node[list,on chain,label={above:7}] (G) {1};
\node[list,on chain,label={above:8}] (H) {1};
\node[list,on chain,label={above:9}] (I) {1};
\node[list,on chain,label={above:10}] (J) {1};
\node[list,on chain,label={above:11}] (K) {0};
\node[list,on chain,label={above:12}] (L) {0};
\node[list,on chain,label={above:13}] (M) {0};
\node[list,on chain,label={above:14}] (N) {1};
\node[list,on chain,label={above:15}] (O) {0};
\node[list,on chain,label={above:16}] (P) {0};
\node[list,on chain,label={above:17}] (Q) {0};
\node[list,on chain,label={above:18}] (R) {0};
\node[list,on chain,label={above:19}] (S) {0};
\end{tikzpicture}
\caption{A binary tree and its level-order bitmap representation.}
\label{fig:succinct-binary-tree} 
\end{figure}
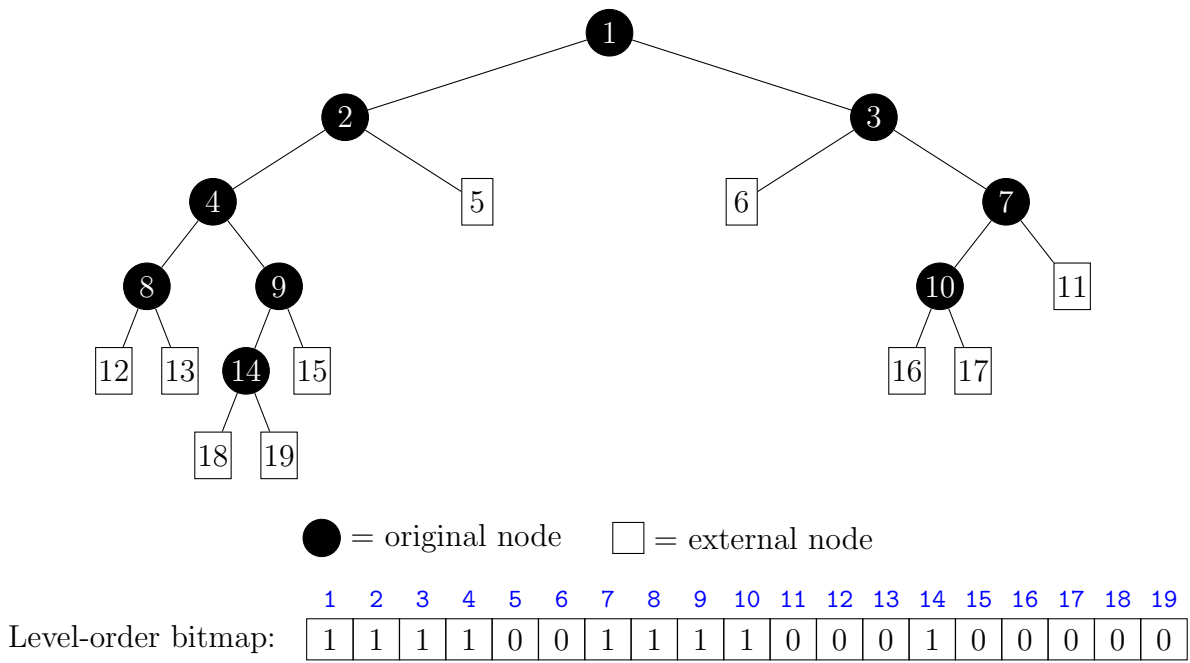

Practical succinct representations of binary trees supporting efficient navigation date back to Jacobson~\cite{Jacobson1989}, who encoded a tree by storing the binary string of length $2n+1$ obtained by adding \emph{external} vertices so that every node has exactly two children, then taking a level-order traversal of the tree and recording a 1 for each original \emph{internal} node encountered and a 0 for each external node encountered~(see Figure \ref{fig:succinct-binary-tree}). If each node is labelled by its position in a level-order traversal then, for instance, the children of the node labelled $x$ in the tree encoded by the string $\sigma$ have labels $2\rank_x(\sigma)$ and $2\rank_x(\sigma)+1$, where $\rank_x(\sigma)$ is the number of ones in $\sigma$ up to (and including) the position $x$. By storing $o(n)$ bits, the rank operation (and similar supporting operations used to retrieve information about the trees) can be implemented in $O(1)$ time. Jacobson's results allow finding a parent or child using $O(\log_2 n)$ bit inspections; Clark \cite{ClarkDavid1997} and Munro \cite{DBLP:conf/fsttcs/Munro96} improved this to $O(1)$ inspections of $\log_2 n$ bit words.

\subsection{Succinct Representations of Tame Tree Classes}

Let $B$ be a function satisfying $B(n) =\Theta(\log n)$. In \cite{MunroNicholsonBenknerWild2021} the authors construct a succinct\footnote{The lower bound of $\log_2 |\mC_n|$ space corresponds to trees that are sampled from a \emph{uniform} distribution. By adapting compression methods for trees, the data structures of Munro et al. actually use less than $\log_2 |\mC_n|$ bits of space for families of trees that are `repetitive' (i.e., have distributions dissimilar to the uniform distribution). Hence, Munro et al. call their data structures \emph{hypersuccinct}.} data structure for any class of (rooted planar) binary trees ${\cal T}$ satisfying the following four conditions.  
\begin{enumerate}
\item \emph{Fringe-hereditary}: For any tree $\tau \in \cal {\cal T}$ and node $v \in \tau$, the \emph{fringe subtree} $\tau[v]$, which consists of $v$ and all of its descendants in $\tau$, also belongs to $\cal T$. 
\item \emph{Worst-case $B$-fringe dominated}: Most nodes in members of ${\cal T}$ do not generate large fringe subtrees, in the sense that 
\[ {\Big|}\{v\in \tau : {\big|}\tau[v]{\big|} \;\geq\; B(n)\} {\Big|} = o(n/ \log B(n)) \]
for every binary tree $\tau$ in the subset ${\cal T}_n \subset\cal T$ containing the members of $\cal T$ with $n$ nodes, where $|\tau|$ denotes the number of nodes in $\tau$ and the constant in the little-$o$ term is independent of $\tau$. 
\item \emph{Log-linear}: There is a constant $c > 0$ and a function $\vartheta(n)=o(n)$ such that 
\begin{equation} 
\log_2|{\cal T}_n| = c n + \vartheta (n).
\label{eq:loglinear}
\end{equation}
\item \emph{$B$-heavy twigged}: If $v$ is a node of any $\tau \in {\cal T}$ with $|\tau[v]| \geq B(n)$, and $\tau_\ell[v]$ and $\tau_r[v]$ are the left and right subtrees of $v$ in $\tau$, then $|\tau_\ell[v]|, |\tau_r[v]| = \omega(1)$.
\end{enumerate} 

A class of binary trees is called \emph{tame} if it satisfies these properties.

\begin{table}[ht]
\begin{centering}
\begin{tabular}{|ll|}
\hline
\texttt{parent}$(v)$ & the parent of $v$, same as \texttt{anc}$(v, 1)$\\
\texttt{degree}$(v)$ & the number of children of $v$\\
\texttt{left\_child}$(v)$ & the left child of node $v$\\
\texttt{right\_child}$(v)$ &the right child of node $v$\\
\texttt{depth}$(v)$ & the depth of $v$, i.e., the number of edges between the root and $v$\\
\texttt{anc}$(v, i)$ & the ancestor of node $v$ at depth $\texttt{depth}(v) - i$\\
\texttt{nbdesc}$(v)$ & the number of descendants of $v$\\
\texttt{height}$(v)$ & the height of the subtree rooted at node $v$\\
\texttt{LCA}$(v, u)$ & the lowest common ancestor of nodes $u$ and $v$\\
\texttt{leftmost\_leaf}$(v)$ & the leftmost leaf descendant of $v$\\
\texttt{rightmost\_leaf}$(v)$ & the rightmost leaf descendant of $v$\\
\texttt{level\_leftmost}$(\ell)$ & the leftmost node on level $\ell$\\
\texttt{level\_rightmost}$(\ell)$ & the rightmost node on level $\ell$\\
\texttt{level\_pred}$(v)$ & the node immediately to the left of $v$ on the same level\\
\texttt{level\_succ}$(v)$ & the node immediately to the right of $v$ on the same level\\
\texttt{node\_rank$_X$}$(v)$ & the position of $v$ in the $X$-order, 
$X \in \{\texttt{PRE, POST, IN}\}$,\\ 
&i.e., in a preorder, postorder, or inorder traversal of the tree\\
\texttt{node\_select$_X$}$(i)$ & the $i$th node in the $X$-order, 
$X \in \{\texttt{PRE, POST, IN}\}$\\
\texttt{leaf\_rank}$(v)$ & the number of leaves before and including $v$ in preorder\\
\texttt{leaf\_select}$(i)$ & the $i$th leaf in preorder\\
\hline
\end{tabular}
\caption{Operations discussed in \cite{He2007, MunroNicholsonBenknerWild2021} which can be done in $O(1)$ time in the $(\log n)$-bit word RAM model in a succinct encoding of a binary tree.  } 
\label{tab:ops}
\end{centering}
\end{table}

\begin{theorem}[{Munro et al.~\cite{MunroNicholsonBenknerWild2021}}] \label{thm:succinct_old}
If $\cal T$ is a tame class of binary trees then there exists a succinct encoding for $\cal T$ that supports the operations in Table~\ref{tab:ops} in $O(1)$ time using the $(\log n)$-bit word RAM model.
\end{theorem}

As described in Section~\ref{sec:our_results} below, in this paper we prove that Theorem~\ref{thm:succinct_old} holds under a weakening of its conditions (see Theorem~\ref{thm:succinct}). 

\subsection{Enumeration of Tree Classes}

In order to deduce the space complexity of a succinct encoding of a family of objects $\mC$, it is necessary to enumerate the number of objects $c_n$ of size $n$ in $\mC_n$. More precisely, a succinct encoding of $\mC$ will use $(\log_2 \rho)n + o(n)$ bits where $\rho = \limsup_{n\rightarrow\infty}|c_n|^{1/n}$ is the (finite in our cases) \emph{exponential growth} of $c_n$. An extremely powerful tool for enumeration is the \emph{generating function}
\[ C(z) = \sum_{n \geq 0}c_nz^n, \]
whose series coefficients are the number of objects of size $n$ in $\mC$. Although the theory of formal power series is well-understood, in this paper all generating functions converge in a neighbourhood of the origin in $\C$, where they represent complex analytic functions. The theory of \emph{analytic combinatorics}~\cite{FlajoletSedgewick2009,Melczer2021} allows one to -- automatically, in many cases -- translate information about the analytic properties of $C(z)$ into a characterization of asymptotics for $c_n$. The ``first principal'' of analytic combinatorics states that the exponential growth of $c_n$ satisfies $\rho = 1/M,$ where $M$ is the smallest modulus among the singularities of $C$.

The mathematical study of trees dates back to work of Cayley~\cite{Cayley1857,Cayley1859} in the 1850s, which already illustrated the usefulness of generating functions by deriving the generating function
\[ B(z) = \sum_{n \geq 0} b_nz^n = \frac{1-\sqrt{1-4z}}{2z} \]
for the number $b_n$ of binary trees on $n$ nodes, and proving the Catalan expression\footnote{Cayley's original ``expression for the general term''~\cite[Page 378]{Cayley1859} was $b_n = \frac{1(3)(5)\cdots(2n-3)}{1(2)(3)\cdots(n)}2^{n-1}$. At that time, Cayley seemed unaware of other combinatorial interpretations of the Catalan numbers.}
\[ b_n = \frac{1}{n+1}\binom{2n}{n}. \]

Generating functions are a powerful tool for the enumeration of tree classes, and the optimal bitsize needed to encode a family of trees that can be recursively decomposed into \emph{independent} subtrees of the same type can often be enumerated automatically. For instance, if $\Omega \subset \N$ is any (not necessarily finite) set of natural numbers containing $0$ then recursively decomposing a tree implies that the generating function $T_{\Omega}(z)$ for the class of trees where every node has a number of children in $\Omega$ satisfies
\[ T_{\Omega}(z) = z \phi\left(T_{\Omega}(z)\right),  \]
where $\phi(t) = \sum_{a \in \Omega}t^a$ encodes the elements of $\Omega$. In simple cases Lagrange inversion~\cite{Gessel2016} can give a closed form for the number of such trees, but even in general an analytic study of this equation implies~\cite[Proposition IV.5]{FlajoletSedgewick2009} the exponential growth $\rho=\phi(\tau)/\tau$, where $\tau$ is the smallest positive solution of
\[ \frac{t\phi'(t)}{\phi(t)} = 1. \]

When the subtrees of a node are \emph{dependent}, as they are for height balanced trees, things get harder. For instance, 2-3 trees (which are $B$-trees of order 3) were introduced by John Hopcroft in unpublished work from 1970. A 2-3 tree can be viewed as a binary tree whose non-leaf nodes have degrees 2 or 3 with the \emph{additional constraint} that all leaves have the same height. While studying different representations of 2-3 trees, Miller et al.~\cite{MillerPippengerRosenbergSnyder1979} gave the recursive formula
\[ p_n = \sum_{2k+3m=n}\binom{k+m}{k}p_{k+m} \]
for the number $p_n$ of 2-3 trees with $n$ leaves, from which they deduced that the exponential growth of $p_n$ is the golden ration $\phi=(1+\sqrt{5})/2$. A few years later, Odlyzko~\cite{Odlyzko1982} used the functional equation
\[ P(z) = z + P(z^2+z^3), \]
which uniquely characterizes the generating function $P(z)$ of $p_n$, to perform a detailed singularity analysis and prove that
\[ p_n \sim \frac{\phi^n}{n}u(\log n) \]
for a $\log(4-\phi)$-periodic analytic function $u$. In fact, Odlyzko derived the asymptotic behaviour for any sequence (under some mild conditions) whose generating function satisfies a functional equation of the form $P(z) = a(z) + P(b(z))$ for polynomials $a$ and $b$ with real nonnegative coefficients; this includes the sequences enumerating $B$-trees of order $m$ by number of leaves, for any $m\geq3$, and the sequence enumerating red-black trees. Further work on functional equations of this form, and their applications, can be found in Teufl~\cite{Teufl2007}.

\subsubsection{Enumerating AVL Trees}

Recall from above that \emph{AVL trees}~\cite{Adelson-VelskiiLandis1962} have balancing rules that force the subtrees rooted at the children of any node differ in height by at most one. Throughout this paper we let $\mA$ denote the class of AVL trees, so that $\mA_n$ consists of all binary trees on $n$ vertices such that the subtrees of any vertex differ in height by at most one (including empty subtrees). Let $a_n = |\mA_n|$ be the counting sequence of $\mA$ and let $A(z) = \sum_{n \geq 0}a_nz^n$ be its associated generating function.  

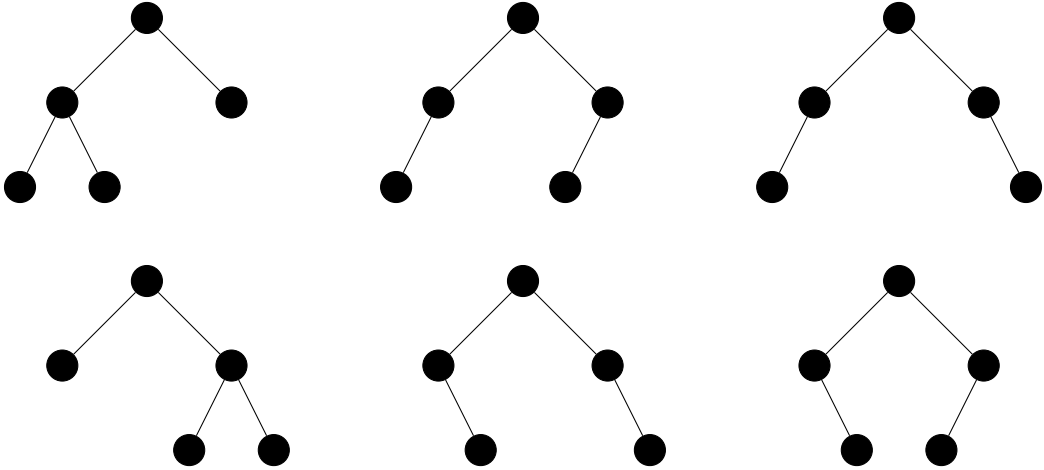
\begin{figure}[t] \centering
\begin{tikzpicture}[level distance=32pt,
every node/.style={circle, draw, inner sep=1pt, fill=black, minimum size=1em},
dummy/.style={rectangle,draw=none,minimum width=1em},
level 1/.style={sibling distance=64pt},
level 2/.style={sibling distance=32pt},
level 3/.style={sibling distance=16pt}
]
\node {}
    child {node {} 
        child {node {}}
        child {node {}}
    }
    child {node {}}
;
\node at (5,0) {}
    child {node {} 
        child {node {}}
        child {edge from parent[draw=none]}
    }
    child {node {}
        child {node {}}
        child {edge from parent[draw=none]}
    }
;
\node at (10,0) {}
    child {node {} 
        child {node {}}
        child {edge from parent[draw=none]}
    }
    child {node {}
        child {edge from parent[draw=none]}
        child {node {}}
    }
;
\node at (0, -3.5) {}
    child {node {}}
    child {node {} 
        child {node {}}
        child {node {}}
    }
;
\node at (5, -3.5) {}
    child {node {}
        child {edge from parent[draw=none]}
        child {node {}}
    }
    child {node {} 
        child {edge from parent[draw=none]}
        child {node {}}
    }
;
\node at (10, -3.5) {}
    child {node {}
        child {edge from parent[draw=none]}
        child {node {}}
    }
    child {node {} 
        child {node {}}
        child {edge from parent[draw=none]}
    }
;
\end{tikzpicture}
\caption{The six types of AVL trees with $n=5$ nodes.}
\label{fig:avltrees-5}
\end{figure}

Because the subtrees of a node in an AVL tree are restricted by height, we let $A_h(z)$ be the generating function for the subclass of AVL trees with height $h$. The balance condition on subtrees implies that an AVL tree of height $h+2$ is a root together with a subtree of height $h+1$ and a subtree of height either $h+1$ or $h$, giving rise to the recursive equation 

\begin{equation} 
A_{h+2}(z)=A_{h+1}(z)(A_{h+1}(z)+2A_h(z))
\label{eq:AVLrec} 
\end{equation}
for all $h \geq 0$, where the factor of $2$ indicates that the shorter subtree can be on the left or right side. This recursion, along with the initial conditions $A_0(z)=z$ (encoding the only AVL tree with height zero, which is a single vertex) and $A_1(z)=z^2$ (encoding the only AVL tree with height one, which is a root with two children), uniquely determines $A_h(z)$ for all $h$. 

Due to the way they are constructed, and because the height of an AVL tree is the worst-case cost to search for one of its elements, AVL trees have previously been studied under height restrictions. Any binary tree with $n$ vertices has height at least $\log_2(n+1)$, and the original paper introducing AVL trees proved~\cite[Lemma 1]{Adelson-VelskiiLandis1962} that an AVL tree with $n$ vertices has height at most $\log_{\frac{1+\sqrt{5}}{2}}(n+1) < (3/2)\log_2(n+1)$. If $e_h=A_h(1)$ is the number of AVL trees of height $h$ then~\eqref{eq:AVLrec} gives the non-linear recurrence
\[ e_{h+2} = e_{h+1}^2 + 2e_{h+1}e_h, \]
with initial conditions $e_0=e_1=1$. Solving this recurrence gives the doubly exponential growth $e_h \sim \lfloor \theta^{2^h}\rfloor$ for a constant $\theta = 1.436\dots$ (see Knuth~\cite[Exercise 6.2.3.7]{Knuth1998a} or Aho and Sloane~\cite{AhoSloane1973}), while the average number of vertices in an AVL tree of height $h$ is $B'_h(1)/B_h(1) \sim \beta2^h$ for a constant $\beta=0.701\dots$ (see Knuth~\cite[Exercise 6.2.3.8]{Knuth1998a} or Khizder~\cite{Khizder1966}).

To enumerate AVL trees by number of vertices, as we must to characterize the space complexity of a succinct encoding, we sum over all possible heights to get the generating function
\[A(z)=\sum_{h=0}^\infty A_h(z).\]
The fact that recurrence~\eqref{eq:AVLrec} involves height-restricted AVL trees, while we sum over all heights to obtain $A(z)$, makes an analysis of the number of AVL trees difficult. In a survey following his tour-de-force asymptotic analysis of $B$-trees, Odlyzko~\cite{Odlyzko1984} stated that the generating function of AVL trees `appears not to satisfy any simple functional equation, but by an intensive study$\dots$it can be shown' that $|\mA_n| \sim n^{-1}\alpha^{-n}u(\log n)$ where $\alpha = 0.5219\dots$ is `a certain constant' and $u$ is a periodic function, referencing for details a paper that was planned to be published but was never written\footnote{The current authors thank Andrew Odlyzko for discussions on the asymptotic behaviour of AVL trees and the growth constant $\alpha$.}. As described in Theorem~\ref{AVL} below, we provide the first published proof of the growth constant $\alpha$ (and a general method to derive the exponential growth for sequences whose generating functions satisfy similar functional equations).

\section{Our Results}
\label{sec:our_results}

In this work\footnote{A shortened version of the present article was published in the 35th International Conference on Probabilistic, Combinatorial and Asymptotic Methods for the Analysis of Algorithms (AofA 2024)~\cite{ChizewerMelczerMunroPun2024}.} we present a new construction that gives a succinct encoding for all classes of trees satisfying only the first three conditions for tameness. By using constant time rank and select operations already supported by a succinct encoding for binary trees, we can also eliminate the use of so-called ``portal nodes'' and thus relax the second condition to the following.
\begin{enumerate}
\item[2$^\prime$.\hspace{-0.5em} ] \emph{Worst-case weakly fringe dominated:} Most nodes in members of ${\cal T}$ do not generate large fringe subtrees, in the sense that there is a $B^\prime(n)$ satisfying $B^\prime(n) = d\log_2 n +o(\log n)$, for some $d<1$, such that
\begin{equation}
{\Big|}\{v\in \tau : {\big|}\tau[v]{\big|} \;\geq\; B^\prime(n)\} {\Big|} = o(n)
\label{eq:weakfringe}
\end{equation}
for every binary tree $\tau \in {\cal T}_n$. 
\end{enumerate}

We call a class of binary trees \emph{weakly tame} if it is fringe-hereditary, worst-case weakly fringe dominated, and log-linear.  

\begin{theorem}\label{thm:succinct}
If $\cal T$ is a weakly tame class of binary trees then there exists a succinct encoding for $\cal T$ that supports the operations in Table~\ref{tab:ops} in $O(1)$ time using the $(\log n)$-bit word RAM model.
\end{theorem}

\begin{proof}
See Section~\ref{sec:encoding}.
\end{proof}

\begin{remark}
We support operations on static trees, leaving extensions to trees with updates (such as in~\cite{DBLP:conf/soda/MunroRS01}) to future work.
\end{remark}

As an immediately corollary, we obtain a succinct data structure for AVL trees.

\begin{corollary}\label{coro:AVL}
There exists a succinct encoding for AVL trees that supports the operations in Table~\ref{tab:ops} in $O(1)$ time using the $(\log n)$-bit word RAM model.
\end{corollary}

Corollary~\ref{coro:AVL} was noted in~\cite[Example F.2]{MunroNicholsonBenknerWild2021} (under the stronger conditions of Theorem~\ref{thm:succinct_old}), however those authors inferred log-linearity of AVL trees from the unproved statement in Odlyzko \cite{Odlyzko1984}. Here we give a rigorous derivation of the exponential growth for the number of AVL trees. Indeed, Equation~\eqref{eq:AVLrec} implies that $A_h(z)$ is a non-constant polynomial with positive coefficients for all $h$, so the equation $A_h(z)=1/3$ has a unique positive solution for all $h \in \mathbb{N}$ (see Figure~\ref{fig:alphas} for values of these solutions). 

\begin{theorem}
\label{AVL}
If $\alpha_h$ is the unique positive solution to $A_h(z)=1/3$ then the limit 
\[ \alpha=\lim_{h\to\infty}\alpha_h=0.5219\ldots\] 
exists. Furthermore, 
\[\log_2(a_n) = \underbrace{n\log_2(\alpha^{-1})}_{n(0.938\ldots)}  + \log_2 s(n) \]
for a function $s$ growing at most sub-exponentially (meaning $s(n) = o(\kappa^n)$ for all $\kappa>1$).
\end{theorem}

\begin{proof}
The result follows immediately from applying Theorem~\ref{thm:general-theo} below with $f(x_1,x_2)=x_1^2+2x_1x_2$, since the unique positive solution to $f(C,C)=C$ is $C=1/3$.
\end{proof}

\begin{figure}
\centering
\includegraphics[width=0.9\linewidth]{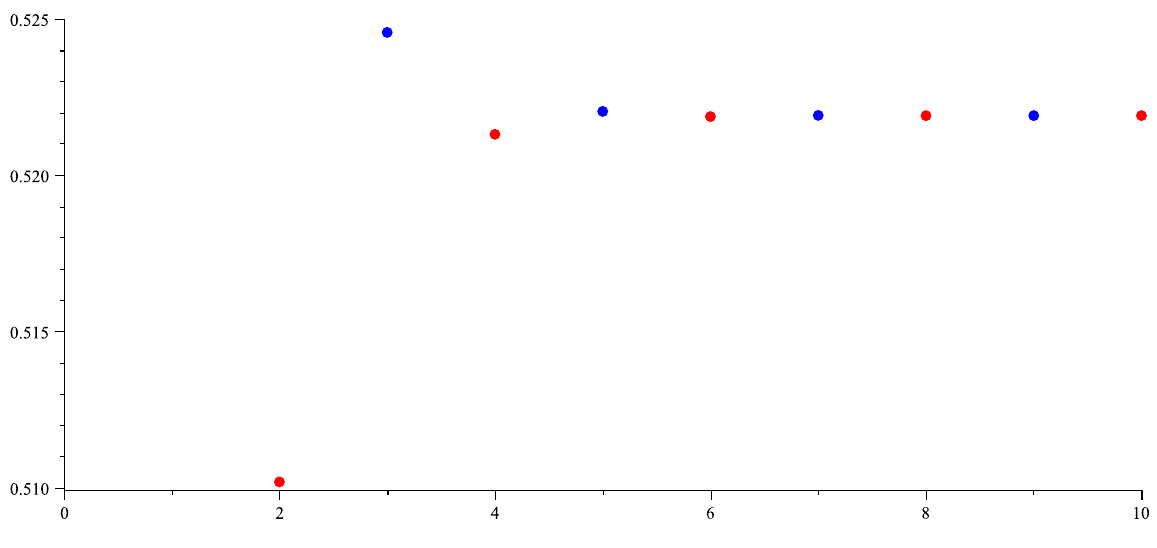}
\caption{Values $\alpha_h$ converging to $\alpha=0.5219\ldots$ monotonically from below among even $h$ (red) and monotonically from above among odd $h$ (blue). }
\label{fig:alphas}
\end{figure}

\begin{remark}
A full proof of the claimed asymptotic behaviour $a_n \sim n^{-1}\alpha^{-n}u(\log n)$ in Odlyzko~\cite{Odlyzko1984}, which characterizes \emph{sub-dominant} asymptotic terms for the bitsize, requires a more intense study of the recursion~\eqref{eq:AVLrec} and is outside the scope of this discussion. It is postponed to future work.
\end{remark}

Our approach derives asymptotics for a family of generating functions satisfying recursive equations similar to~\eqref{eq:AVLrec}. Indeed, inspired by the work of Sedgewick~\cite{Sedgewick2008} on Left-Leaning Red Black Trees, we also study the class of \emph{Left-Leaning AVL (LLAVL) Trees}, which are AVL trees with the added restriction that at every node the height of the left subtree is at least the height of the right subtree. If $L_h(z)$ is the generating function for LLAVL trees with height $h$ then
\begin{align}
L_{h+2}(z)=L_{h+1}(z)(L_{h+1}(z)+L_{h}(z))
\label{eq:LLAVLrec} 
\end{align}
for all $h\geq 0$, as an LLAVL tree of height $h+2$ is a root together with a left subtree of height $h+1$ and a right subtree of height $h+1$ or $h$. Note that the only difference between this recurrence and the recursive equation~\eqref{eq:AVLrec} for AVL trees is the coefficient of $L_h(z)$, since there is now only one way to have an unbalanced pair of subtrees.

\begin{theorem}
\label{LLAVL}
If $\gamma_h$ is the unique positive solution to $L_h(z)=1/2$ then the limit 
\[ \gamma=\lim_{h\to\infty}\gamma_h=0.67418\ldots\] 
is well-defined. Furthermore, the number $\ell_n$ of LLAVL trees on $n$ nodes satisfies
\[\log_2(\ell_n) = \underbrace{n\log_2(\gamma^{-1})}_{n(0.568\ldots)}  + \log_2 s(n) \]
for a function $s$ growing at most sub-exponentially.
\end{theorem}

\begin{proof}
The result follows by applying Theorem~\ref{thm:general-theo} below with ${f(x_1,x_2)=x_1^2+x_1x_2}$, since the unique positive solution to $f(C,C)=C$ is $C=1/2$.
\end{proof}

LLAVL trees are also weakly tame, giving the following.

\begin{corollary}\label{coro:LLAVL}
There exists a succinct encoding for LLAVL trees that supports the operations in Table~\ref{tab:ops} in $O(1)$ time using the $(\log n)$-bit word RAM model.
\end{corollary}

The remainder of this paper proves our main results.

\section{A New Succinct Encoding for Weakly Tame Classes}
\label{sec:encoding}

We begin by proving Theorem~\ref{thm:succinct}, first describing our encoding and then showing it has the stated properties.

\subsection{Our encoding}
 
Let $\cal E$ denote a succinct data structure representing all binary trees that supports the operations in Table~\ref{tab:ops}, and denote the encoding of a binary tree $\tau$ in this data structure by ${\cal E}(\tau)$. We now fix a weakly tame class of binary trees $\cal T$ and, given a binary tree $\tau\in \cal T$ of size $n$, define the \emph{upper tree}
\[\tau' =  {\Big\{}v\in \tau:  {\big|}\tau[p(v)]{\big|} \geq d\log_2 n {\Big\}} \] 
where $p(v)$ denotes the parent of a vertex $v$ in the tree $\tau$ and $d$ is a constant such that $B^\prime(n) = d\log_2 n +o(\log n)$ satisfies~\eqref{eq:weakfringe} in the definition of worst-case weakly fringe dominated. 

\begin{figure}\center
\begin{tikzpicture}[level distance=25pt,
every node/.style={circle,draw,inner sep=1pt, minimum size=0.7em},
level 1/.style={sibling distance=192pt},
level 2/.style={sibling distance=96pt},
level 3/.style={sibling distance=48pt},
level 4/.style={sibling distance=24pt},
level 5/.style={sibling distance=12pt}
]
\node[fill=red] {}
    child {node[fill=red] {} 
        child {node[fill=red] {}
            child {node[fill=blue] {}
                        child {node[fill=black] {}}
                child {node[fill=black] {}
                            child {node[fill=black] {}}
                    child {node[fill=black] {}}
                    }
                    }
            child {node[fill=blue] {}
                child {edge from parent[draw=none]}
                child {node[fill=black] {}}
                }
        }
                child {node[fill=blue] {}
            child {node[fill=black] {}
                        child {node[fill=black] {}}
                child {node[fill=black] {}}
                }
            child {node[fill=black] {}}
        }
    }
    child {node[fill=red] {}
        child {node[fill=red] {}
            child {node[fill=blue] {}
                child {node[fill=black] {}
                            child {node[fill=black] {}}
                    child {node[fill=black] {}}
                    }
                child {node[fill=black] {}}
            }
            child {node[fill=blue] {}
                child {node[fill=black] {}}
                child {node[fill=black] {}
                    child {node[fill=black] {}}
                    child {node[fill=black] {}}
                }
            }
        }
        child {node[fill=blue] {}
            child {node[fill=black] {}}
            child {node[fill=black] {}
                child {node[fill=black] {}}
                child {node[fill=black] {}}
            }
        }
    }
;
\end{tikzpicture}\\[1em]

        $\tau = $\, \begin{tikzpicture}
        [level distance=25pt,
every node/.style={circle,draw,inner sep=1pt, minimum size=0.7em},
level 1/.style={sibling distance=96pt},
level 2/.style={sibling distance=48pt},
level 3/.style={sibling distance=24pt},
level 4/.style={sibling distance=12pt},
level 5/.style={sibling distance=12pt}
]
\node[fill=red] {}
    child {node[fill=red] {} 
        child {node[fill=red] {}
            child {node[fill=blue] {}}
            child {node[fill=blue] {}}
        }
                child {node[fill=blue] {}
        }
    }
    child {node[fill=red] {}
        child {node[fill=red] {}
            child {node[fill=blue] {}}
            child {node[fill=blue] {}}
        }
        child {node[fill=blue] {}}
    }
;
\end{tikzpicture}\hfill
        $00 = $\, \begin{tikzpicture}[level distance=25pt,
every node/.style={circle,draw,inner sep=1pt, minimum size=0.7em},
level 1/.style={sibling distance=24pt},
level 2/.style={sibling distance=12pt},
]
\node[fill=blue] {}
    child {node[fill=black] {}}
    child {node[fill=black] {}
        child {node[fill=black] {}}
        child {node[fill=black] {}}
    }
;
\end{tikzpicture}\hfill
$01 = $\,       \begin{tikzpicture}[level distance=25pt,
every node/.style={circle,draw,inner sep=1pt, minimum size=0.7em},
level 1/.style={sibling distance=24pt},
level 2/.style={sibling distance=12pt},
]
\node[fill=blue] {}
    child {edge from parent[draw=none]}
    child {node[fill=black] {}}
;
\end{tikzpicture}\hfill
$10 = $\,       \begin{tikzpicture}[level distance=25pt,
every node/.style={circle,draw,inner sep=1pt, minimum size=0.7em},
level 1/.style={sibling distance=24pt},
level 2/.style={sibling distance=12pt},
]
\node[fill=blue] {}
    child {node[fill=black] {}
        child {node[fill=black] {}}
        child {node[fill=black] {}}
    }
    child {node[fill=black] {}}
;
\end{tikzpicture}\\[1em]
\caption{An illustration of our succinct encoding for weakly tame tree classes, including the names of the shapes of the subtrees rooted at leaves (in blue) of the upper tree $\tau$ (in red and blue).}
\label{fig:encoding}
\end{figure}

Our succinct data structure for $\cal T$ is constructed as follows. 
\begin{enumerate}
    \item We simply copy the encodings ${\cal E}(\tau')$ for upper trees. 
    \item For every $1\leq j< d\log_2 n$ we write down a lookup table mapping the trees in ${\cal T}_j$ (with $j$ nodes) to their corresponding $\cal E$ encoding. We can do this, for example, by enumerating the ${\cal T}_j$ in lexicographic order by the $\cal E$ encoding using integers of bitsize $\log_2 |{\cal T}_j|=cj+o(j)$, where $c$ is the constant in the definition of log-linearity~\eqref{eq:loglinear}.
    \item For each leaf node $\ell \in \tau'$ the fringe subtree $\tau[\ell]$ has size $|\tau[\ell]|< d\log_2 n$ by definition of $\tau'$. We call these trees \emph{lower trees}, and write them down using their encoding in a lookup table in \texttt{leaf\_rank} order of their roots in $\tau'$, storing the root locations in an indexable dictionary.  
    \item Lastly, we store additional information in (fully) indexable dictionaries 
    to support operations like \texttt{node\_rank/select}, \texttt{level\_succ/pred}, and \texttt{leaf\_rank/select}. For instance, for \texttt{node\_rank/select} we store a fully indexable dictionary that maps the \texttt{node\_rank} for a node in $\tau'$ to the \texttt{node\_rank} of the node in $\tau$. The techniques to support the other operations are similar, and are analogous to constructions used in \cite{He2007, Farzan2014}.
\end{enumerate}

\subsection{Proof of Size and Operation Time Bounds}
Navigation through the upper tree follows standard navigation using $\cal E$, which supports the desired operations in constant time. When a leaf node $\ell$ is reached in the upper tree, the operation $x= \texttt{leaf\_rank}(\ell)$ gives the index of the child tree in the indexable dictionary. Then the operation $\texttt{select}(x)$ gives the location of the string encoding the child tree. Finally, using the table mapping our encoding to the $\cal E$ encoding gives us the ability to perform all the navigation operations on the smaller tree. In order to perform the lookup using the mapping, it is necessary to know the size of the tree. This can be inferred from the space in memory allocated to the naming, which can be calculated by the operation $\texttt{select}(x+1)$ in the indexable dictionary to find the starting location of the next child tree. To navigate back to the upper tree from a child tree, we use the reverse operations of $y=\texttt{rank}(x)$ in the indexable dictionary followed by $\texttt{select\_leaf}(y)$ in the upper tree. 

To get the \texttt{node\_rank} of a node in $\tau'$ we use the fully indexable dictionary, and to get the \texttt{node\_rank} of a node not in $\tau'$ we simply get the \texttt{node\_rank} of the root of the child tree and the \texttt{node\_rank} of the node within the child tree and perform the appropriate arithmetic depending on the desired rank order (\texttt{pre, post, in}). For \texttt{node\_select}, if the node is in $\tau'$ then selecting using the indexable dictionary is sufficient. Otherwise, the node is in a child tree and the initial \texttt{node\_select} will return the predecessor node in $\tau'$ which will be the root of the child tree when using \texttt{preorder} (the argument is similar for \texttt{postorder} and \texttt{inorder}). Using the rank of this root and appropriate arithmetic, we can then select the desired node in the child tree. Implementing the other operations is analogous. It is clear that all of these operations are supported in constant time, since they involve a constant number of calls to the constant-time operations in the existing data structures, and lookups using $(\log n)$-bit words. 

\paragraph*{Space Complexity}
The space used by ${\cal E}(\tau')$ is $o(n)$ by the weakly tame property. The space used by the lookup tables is $O(n^d \log n)=o(n)$ by definition of $\tau'$ and $d$, and the space used by all of the encodings of the child trees is $cn+o(n)$ by log-linearity. Lastly, the space needed for the indexable dictionaries is $o(n)$ for each~\cite[Lemmas 1 and 2]{Farzan2014}. Summing these requirements shows that the total storage required is $cn +o(n)$ many bits, so the encoding is succinct.\qed

\section{Asymptotics for a Family of Recursions}\label{sec:asymp}

As noted above, we derive the asymptotic behaviour of a family of generating functions which includes Theorem~\ref{AVL} as a special case. Let $\mF$ be a combinatorial class decomposed into a disjoint union of finite subclasses $\mF=\bigcup_{h=0}^\infty\mF_h$ whose generating functions $F_h(z)$ are non-constant and satisfy a recursion
\begin{equation}
F_h(z) = f(F_{h-1}(z), F_{h-2}(z),\dots, F_{h-c}(z)) \quad \text{for all} \quad h \geq c,
\label{eq:Frec}
\end{equation} 
where $c$ is a positive integer and $f$ is a multivariate polynomial with non-negative coefficients.

\begin{remark}
 The elements of $\mF_h$ are usually \emph{not} the objects of $\mF$ of size $h$ (in our tree applications they contain trees of height $h$, not trees with $h$ nodes). 
 The fact that each $\mF_h$ is finite implies that the $F_h(z)$ are polynomials with non-negative coefficients. The coefficient of $z^n$ in $F_h(z)$ counts the number of objects of size $n$ within the subclass indexed by $h$ (i.e., the number of trees with $n$ nodes and height $h$ in our applications).
\end{remark}

We assume that there exists a (necessarily unique) positive real solution $C$ to the equation $C=f(C,C,\dots, C)$, which we call a \emph{fixed point} of $f$, and for each $h\geq0$ we let $\alpha_h$ be the unique positive real solution to $F_h(z)=C$. In order to rule out degenerate cases and cases where the counting sequence has periodic behaviour, we need another definition.

\begin{definition}(recursive-dependent)
\label{as:recdep}
We call the polynomial $f$ \emph{recursive-dependent} if there exists a constant $k$ (depending only on $f$) such that for any indices $i,j\geq c$ with $i\geq j+k$ there exists a sequence of applications of the recurrence~\eqref{eq:Frec} resulting in a polynomial $P$ with $F_i = P(F_{\ell_1},\dots,F_{\ell_m})$ for some $0\leq \ell_1<\dots<\ell_m\leq i$ where $\frac{\partial{P}}{\partial{F_j}}$ is not the zero polynomial. 
\end{definition}

\begin{example}
The polynomial $f(x,y) = y$ is not recursive-dependent because it leads to the recursion $F_h(z)=F_{h-2}(z)$, meaning that the values of $F_h$ when $h$ is even can be independent of those where $h$ is odd. In particular, for any positive integer $n$ we cannot express $F_{2n}$ as a polynomial involving $F_1,F_3,\dots,F_{2n-1}$.
\end{example}

\begin{lemma}\label{lem:alphaexists}
If $f$ is recursive-dependent with non-negative coefficients and a positive fixed point then the limit $\alpha=\lim_{h\to\infty}\alpha_h$ exists. 
\end{lemma}

\begin{figure}
\centering
\includegraphics[width=0.9\linewidth]{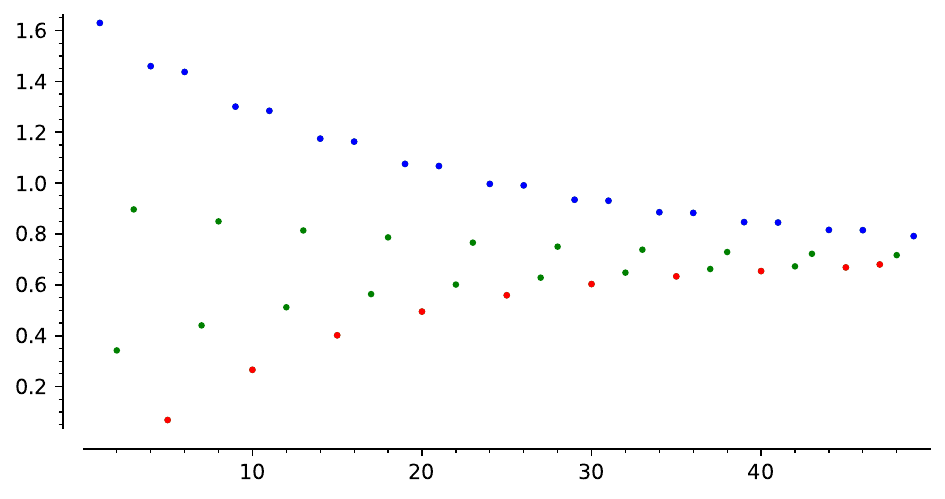}
\caption{Values $\alpha_i$ converging with $u_i$s shown in blue and $\ell_j$s shown in red. }\label{fig:alpha-ul}
\end{figure}

\begin{proof}
We start by defining two subsequences of $\alpha_h$ to give upper and lower bounds on its limit, then prove that these are equal. First, we let
\begin{itemize}
    \item $u_0$ be the smallest index $j \in \{0,\dots,c-1\}$ such that $\alpha_j = \max\{\alpha_0,\dots,\alpha_{c-1}\}$
\end{itemize}
and for all $i\geq0$ let
\begin{itemize}
    \item $u_{i+1}$ be the smallest index $j \in \{u_i+1,\dots,u_i+c\}$ such that $\alpha_j = \max\{\alpha_{u_i+1},\dots,\alpha_{u_i+c}\}$,
\end{itemize}
so that the $u_i$ denote the indices of the maximum values of the $\alpha_h$ as $h$ ranges over intervals of size at most $c$. Conversely, we let
\begin{itemize}
    \item $\ell_0$ be the index $j \in \{0,\dots,c-1\}$ such that $\alpha_j = \min\{\alpha_0,\dots,\alpha_{c-1}\}$
\end{itemize}
and for all $j\geq0$ let
\begin{itemize}
    \item $\ell_{i+1}$ be the index $j \in \{u_i+1,\dots,u_i+c\}$ such that $\alpha_j = \min\{\alpha_{u_i+1},\dots,\alpha_{u_i+c}\}$,
\end{itemize}
so that the $\ell_j$ denote the indices of the minimum values of the $\alpha_h$ as $h$ ranges over intervals of size at most~$c$.

We claim that the subsequence $\alpha_{u_i}$ is non-increasing. To establish this, we fix $i\geq1$ and consider $\alpha_{u_i}$. By definition, $\alpha_{u_i} \geq \alpha_{u_j}$ for all $j \in \{u_{i-1}+1,\dots,u_{i-1}+c\}$. Thus, if $u_{i+1} \in \{u_{i-1}+1,\dots,u_{i-1}+c\}$ then $\alpha_{u_i} \geq \alpha_{u_{i+1}}$ as claimed. If, on the other hand, $u_{i+1} > u_{i-1}+c$ then repeated application of the recursion~\eqref{eq:Frec} implies
\begin{align*}
F_{u_{i+1}}(\alpha_{u_i}) &= f{\Big(}F_{u_{i+1}-1}(\alpha_{u_i}),\dots,F_{u_{i+1}-c}(\alpha_{u_i}){\Big)} \\
& \;\; \vdots \\
&= Q{\Big(}F_{u_{i-1}+1}(\alpha_{u_i}),\dots,F_{u_{i-1}+c}(\alpha_{u_i}){\Big)},
\end{align*}
where $Q$ is a multivariate polynomial with non-negative coefficients such that $Q(C,\dots,C)=C$. All the $F_h$ are monotonically increasing as non-constant polynomials with non-negative coefficients, so $F_j(\alpha_{u_i}) \geq F_j(\alpha_{u_j}) = C$ for all $j \in \{u_{i-1}+1,\dots,u_{i-1}+c\}$ and
\[ F_{u_{i+1}}(\alpha_{u_i}) \geq Q{\Big(}C,\dots,C{\Big)} = C.\]
Since $F_{u_{i+1}}$ is monotonically increasing and $F_{u_{i+1}}(\alpha_{u_{i+1}})=C$, we once again see that $\alpha_{u_i} \geq \alpha_{u_{i+1}}$. As $i$ was arbitrary, we have proven that $\alpha_{u_i}$ is non-increasing. The same argument, reversing inequalities, proves that the subsequence $\alpha_{\ell_j}$ is non-decreasing. 

As $\alpha_{\ell_j}$ is non-decreasing and $\alpha_{u_i}$ is non-increasing, either $\alpha_{\ell_j}\leq \alpha_{u_i}$ for all $i,j\geq 0$ or $\alpha_{\ell_j} > \alpha_{u_i}$ for all sufficiently large $i$ and $j$. The second case implies the existence of indices $a,b>0$ such that $\alpha_{\ell_b} > \alpha_{u_a}$ but $\ell_b \in \{u_{a-1}+1,\dots,u_{a-1}+c\}$ so that $u_a$ is not the maximum index of $\alpha_j$ in this range, giving a contradiction. Thus, $\alpha_{\ell_j}\leq \alpha_{u_i}$ for all $i,j\geq 0$ and the limits
\[ 
\alpha_u = \lim_{i\rightarrow\infty} \alpha_{u_i} \quad\text{and}\quad \alpha_\ell = \lim_{j\rightarrow\infty} \alpha_{\ell_j} 
\]
exist. To prove that the limit of $\alpha_h$ exists as $h\rightarrow\infty$, it is now sufficient to prove that $\alpha_u = \alpha_\ell$.

Suppose toward contradiction that $\alpha_u\neq \alpha_\ell$, and define $a=\alpha_u-\alpha_\ell>0$. For any $\epsilon > 0$, we pick $i,j,k$ sufficiently large so that $\ell_j>u_i>\ell_k+c$ and $|\alpha_{u_i}-\alpha_u|,|\alpha_{\ell_j}-\alpha_\ell|,|\alpha_{\ell_k}-\alpha_\ell|<\epsilon$. Then by recursive-dependence we can recursively decompose $F_{\ell_j}$ in terms of $F_{u_i}$, and possibly some other terms $F_{h_1},\dots,F_{h_r}$ where each $|h_n-u_i|\leq c$, to get
\[ C = F_{\ell_j}(\alpha_{\ell_j}) = P(F_{u_i}(\alpha_{\ell_j}), F_{h_1}(\alpha_{\ell_j}),\dots, F_{h_r}(\alpha_{\ell_j})) \]
where $P(F_{u_i},F_{h_1},\dots,F_{h_r})$ is a polynomial with non-negative coefficients that depends on $F_{u_i}$ and satisfies $P(C,\dots,C)=C$. Because $P$ is monotonically increasing in each coordinate, and $\alpha_{\ell_k}+\epsilon > \alpha_\ell \geq \alpha_{\ell_j}$, we see that
\[ C \leq P(F_{u_i}(\alpha_{\ell_k}+\epsilon), F_{h_1}(\alpha_{\ell_k}+\epsilon),\dots, F_{h_r}(\alpha_{\ell_k}+\epsilon)). \]
Furthermore, each $\alpha_{h_n}\geq\alpha_{\ell_k}$ so
\begin{align*}
C 
&\leq P(F_{u_i}(\alpha_{\ell_k}+\epsilon), F_{h_1}(\alpha_{h_1}+\epsilon),\dots, F_{h_r}(\alpha_{h_r}+\epsilon)) \\
&\leq P(F_{u_i}(\alpha_{\ell_k}+\epsilon), C + \mathrm{poly}(\epsilon),\dots, C + \mathrm{poly}(\epsilon)).
\end{align*}
Finally, $\alpha_{u_i}-a \geq \alpha_{\ell_k}$ so
\[ C \leq P(F_{u_i}(\alpha_{u_i}-a+\epsilon), C + \mathrm{poly}(\epsilon),\dots, C + \mathrm{poly}(\epsilon)). \]

Because $a$ is fixed, $P$ is monotonically increasing in each variable, and $F_{u_i}(\alpha_{u_i})=C$, taking $\epsilon\rightarrow0$ shows that the right-hand side of this last inequality is strictly less than $P(C,\dots,C)=C$, a contradiction.
Thus, $a=0$ and the limit $\alpha=\alpha_u=\alpha_\ell$ exists.
\end{proof}

\begin{theorem}\label{thm:general-theo}
If $f$ is recursive-dependent with non-negative coefficients and a positive fixed point, then the number $a_n$ of objects in $\mF$ of size $n$ satisfies
$$a_n = \alpha^{-n} \, s(n),$$ 
where $\alpha$ is the limit described in Lemma~\ref{lem:alphaexists} and $s(n)$ is a function growing at most sub-exponentially. 
\end{theorem}

\begin{proof}
We prove that the generating function $F(z)$ is analytic for $|z|<\alpha$ by showing that the series $\sum_{h=0}^\infty F_h(z)$ converges for these values of $z$. Because $|F(z)|\rightarrow\infty$ as $z\rightarrow\alpha$, the point $z=\alpha$ is then a singularity of $F(z)$ of smallest modulus, and thus (by the root test for series convergence) the reciprocal of the exponential growth of $a_n$.

First, assume that there exists some $k\geq 0$ and $0<\lambda<1$ such that $F_h(z)<\lambda C$ for every $h\in \{k,k+1,\dots, k+c-1\}$. Let $A$ be the sum of the coefficients of all degree 1 terms of $f$. Since $f$ has non-negative coefficients and a positive real fixed point, we must have $A<1$. Let $g(x_1,\dots, x_c)$ be the function created by removing all degree one terms from $f$. Observe that $C=AC+g(C,\dots, C)$, and thus $g(\lambda C,\dots, \lambda C)\leq \lambda^2g(C,\dots,C) = \lambda^2(1-A)C$, so that
\[f(\lambda C,\dots, \lambda C)\leq A\lambda C + \lambda^2(1-A)C.\] 
Algebraic manipulation shows that $A \lambda + \lambda^2(1-A)\leq \lambda$, and since $f$ has non-negative coefficients we can conclude that for every $h\in \{k+c,k+1+c,\dots, k+2c-1\}$ we have $F_h(z) \leq A\lambda C + \lambda^2(1-A)C$. Let $\lambda_0=\lambda$ and define $\lambda_i= \lambda_{i-1}(A + \lambda_{i-1}-A\lambda_{i-1})$ for all $i\geq 1$. By the above argument we have 
\[F_{ch+k}(z)\leq \lambda_h C,\] 
so it remains to show that $\sum_{i=0}^\infty \lambda_i$ converges. We will show that $\lambda_i\leq \lambda(A+\lambda-A\lambda)^i$ by induction on $i$. The result holds by definition for $i=1$. If the result holds for some $j\geq 1$ then
\begin{align*}
\lambda_{j+1} &= \lambda_j (A + \lambda_j - A \lambda_j) \\
&\leq \lambda(A+\lambda-A\lambda)^j(A+\lambda_j-A\lambda_j)\\
&\leq\lambda(A+\lambda-A\lambda)^{j+1}, 
\end{align*}
where the last inequality follows from the fact that $\lambda_j<\lambda$ since
$A+\lambda-A\lambda<1$. 
The sum $\sum_{i=0}^\infty \lambda (A+\lambda-A\lambda)^i$ converges as a geometric series, and thus 
$\sum_{h=0}^\infty F_h(z)$ converges. 

It remains to show that if $|z|<\alpha$ then such a $k$ and $\lambda$ exist. For any $|z|<\alpha$ there is some $N$ sufficiently large such $|z|< \alpha_n$ for all $n\geq N$. By the definition of $\alpha_n$, and since the coefficients of $F_n$ are all positive, we must have $F_n(z)<C$. Hence $F_n(z)<\lambda_n C$ for some $0<\lambda_n < 1$. Taking $k=N$ and letting $\lambda$ be the largest $\lambda_n$ for $n\in \{N, N+1,\dots, N+c-1\}$ proves our final claim.
\end{proof}

\section{Acknowledgements}
The authors thank Andrew Odlyzko for discussions on the asymptotic behaviour of AVL trees and the growth constant $\alpha$, thank Sebastian Wild for alerting us to some relevant references, and thank the referees for their comments.

\bibliographystyle{plain}
\bibliography{bibl}

\end{document}